\documentstyle[11pt,twoside]{article}
\def\Cchi{{\raisebox{.2ex}{ \large $\chi$}}}
\newtheorem{prethm}{{\bf Theorem}}

\newenvironment{thm}{\begin{prethm}{\hspace{-0.5
em}{\bf.}}}{\end{prethm}}
\newtheorem{precor}{{\bf Corollary}}

\newenvironment{cor}{\begin{precor}{\hspace{-0.5
em}{\bf.}}}{\end{precor}}
\newtheorem{preprop}{{\bf Proposition}}

\newenvironment{prop}{\begin{preprop}{\hspace{-0.5
em}{\bf.}}}{\end{preprop}}
\newtheorem{preque}{{\bf Question}}

\newtheorem{preques}{{\bf Question}}

\newtheorem{preremark}{{\bf Remark}}

\newenvironment{remark}{\begin{preremark}{\hspace{-0.5
em}{\bf.}}}{\end{preremark}}
\newtheorem{prelemma}{{\bf Lemma}}

\newtheorem{prelemm}{{\bf Lemma}}

\newtheorem{preex}{{\bf Example}}

\newtheorem{prepro}{{\bf Proposition}}

\newtheorem{preobs}{{\bf Observation}}

\newenvironment{obs}{\begin{preobs}{\hspace{-0.5
em}{\bf.}}}{\end{preobs}}
\newtheorem{preprob}{{\bf Problem}}

\newtheorem{prelem}{{\bf Theorem}}

\newtheorem{preproof}{{\bf Proof.}}

\newenvironment{proof}[1]{\begin{preproof}{\rm
               #1}\hfill{$\Box$}}{\end{preproof}}

\newtheorem{preconj}{{\bf Conjecture}}

\newtheorem{predefi}{{\bf Definition}}

\newenvironment{defi}{\begin{predefi}{\hspace{-0.90
em}{\bf~.}}}{\end{predefi}}
\newtheorem{predeff}{{\bf Definition}}

\input{amssym}

\date{}
\title{{\large\bf \vspace*{-5mm} The Locating Chromatic Number of the Join of Graphs}}

{\small
\author{
{\sc Ali Behtoei\footnote{{\small \it alibehtoei@math.iut.ac.ir}}}\\
{\small \it  Department of Mathematical Sciences}\\
{\small \it  Isfahan University of Technology} \\
{\small \it 84156-83111, \ Isfahan, Iran}}

\begin{document}
\small \voffset=-20mm

\maketitle

\begin{abstract}
 Let $f$ be a proper $k$-coloring of a connected graph $G$  and
$\Pi=(V_1,V_2,\ldots,V_k)$ be an ordered partition of $V(G)$ into
the resulting color classes. For a vertex $v$ of $G$, the color
code of $v$ with respect to $\Pi$ is defined to be the ordered
$k$-tuple $c_{{}_\Pi}(v):=(d(v,V_1),d(v,V_2),\ldots,d(v,V_k)),$
where $d(v,V_i)=\min\{d(v,x)~|~x\in V_i\}, 1\leq i\leq k$. If
distinct vertices have distinct color codes, then $f$ is called a
locating coloring.  The minimum number of colors needed in a
locating coloring of $G$ is the locating chromatic number of $G$,
denoted by  $\Cchi_{{}_L}(G)$. In this paper, we study the locating chromatic number of the join of graphs. We show that when $G_1$ and $G_2$ are two connected graphs with diameter at most two, then $\Cchi_{{}_L}(G_1+G_2)=\Cchi_{{}_L}(G_1)+\Cchi_{{}_L}(G_2)$, where $G_1+G_2$ is the join of $G_1$ and $G_2$. Also, we determine the
locating chromatic numbers of the join of paths, cycles and complete multipartite graphs.
\end{abstract}


\section{Introduction}
 Let $G$ be a graph
without loops and multiple edges with vertex set $V(G)$ and edge
set $E(G)$. A proper $k$-coloring of $G$, $k\in \Bbb{N}$,  is a
function $f$ defined from $V(G)$ onto a set of colors
$[k]:=\{1,2,\ldots,k\}$ such that every two adjacent vertices have
different colors. In fact, for every $i$, $1\leq i\leq k$, the set
$f^{-1}(i)$ is a nonempty independent set of vertices which is
called the color class $i$. When $S\subseteq V(G)$, then $f(S):=\{f(u)~|~u\in S\}$. The minimum cardinality $k$ for which
$G$ has a proper $k$-coloring is the chromatic number of $G$,
denoted by $\Cchi(G)$. For a connected graph $G$, the distance
$d(u,v)$ between two vertices $u$ and $v$ in $G$ is the length of
a shortest path between them, and for a subset $S$ of $V(G)$, the
distance between $u$ and $S$ is given by
$d(u,S):=\min\{d(u,x)~|~x\in S\}$. The diameter of $G$ is $\max\{d(u,v)~|~u,v\in V(G)\}$. When $u$ is a vertex of $G$, then the neighbor of $u$ in $G$ is the set $N_G(u):=\{v~|~v\in V(G),~d(u,v)=1\}$.

\begin{defi} {\rm \cite{XL}}
Let $f$ be a proper $k$-coloring of a connected graph $G$ and
$\Pi=(V_1,V_2,\ldots,V_k)$ be an ordered partition of $V(G)$ into
the resulting color classes. For a vertex $v$ of $G$, the {\bf color
code} of $v$ with respect to $\Pi$ is defined to be the ordered
$k$-tuple $$c_{{}_\Pi}(v):=(d(v,V_1),d(v,V_2),\ldots,d(v,V_k)).$$
  If distinct vertices of $G$ have distinct color codes, then $f$ is
called a {\bf locating coloring} of $G$. The {\bf locating chromatic number},
denoted by $\Cchi_{{}_L}(G)$,  is the minimum number of colors in a locating
coloring of $G$.
\end{defi}

The concept of locating coloring was first introduced and studied
by Chartrand et al. in \cite{XL}. They established some bounds for
the locating chromatic number of a connected graph. They also
proved that for a connected graph $G$ with $n\geq 3$ vertices, we
have $\Cchi_{_L}(G)=n$ if and only if $G$ is a complete
multipartite graph. Hence, the locating chromatic number of the
complete graph $K_n$ is $n$. Also for paths and cycles of order
$n\geq 3$ it is proved in \cite{XL} that $\Cchi_{_L}(P_n)=3$,
$\Cchi_{_L}(C_n)=3$ when $n$ is odd, and $\Cchi_{_L}(C_n)=4$ when
$n$ is even. \\The locating chromatic numbers of trees, Kneser
graphs, Cartesian product of graphs, and the amalgamation of stars are studied in \cite{XL},
\cite{Behtoei}, \cite{Behtoei2}, and \cite{Asmiati}, respectively. For more
results in the subject and related subjects, see~\cite{Asmiati}
to \cite{Conditional}.

Obviously, $\Cchi(G)\leq \Cchi_{{}_L}(G)$. Note that the $i$-th
coordinate of the color code of each vertex in the color class
$V_i$ is zero and its other coordinates are non zero. Hence, a
proper coloring is a locating coloring whenever the color codes of vertices
in each color class are different.

Recall that the join of two graphs $G_1$ and $G_2$,
denoted by $G_1+G_2$, is a graph with vertex set $V(G_1)\bigcup V(G_2)$ and edge set $E(G_1)\bigcup E(G_2)\bigcup\{uv~|~u\in V(G_1),~v\in V(G_2)\}$. For example $K_1+P_n$ is the fan $F_n$, $K_1+C_n$ is the wheel $W_n$, and the friendship graph $Fr_n$, $n=2t+1$, is the graph obtained by joining $K_1$ to the $t$ disjoint copies of $K_2$.

In this paper, we study the locating chromatic number of the join of graphs. Although we always have $\Cchi(G_1+G_2)=\Cchi(G_1)+\Cchi(G_2)$, but it may happen that $\Cchi_{{}_L}(G_1+G_2)\neq\Cchi_{{}_L}(G_1)+\Cchi_{{}_L}(G_2)$. For example we have $\Cchi_{{}_L}(P_{10})=3$ while, by Corollary \ref{pnpn} (see Section 3), $\Cchi_{{}_L}(P_{10}+P_{10})=8$.

The diameter of $G_1+G_2$ is at most two. Hence, in each proper coloring of $G_1+G_2$, the color code of no vertex of $G_1+G_2$ has a coordinate greater than two. This fact motivated us to define a new parameter, the neighbor locating chromatic number, which is closely related to the locating chromatic number. Proposition \ref{tasavi} and Theorem \ref{xlxl2} (see Section 3) show the relation of this parameter with the locating chromatic number. Using this new parameter we determine the exact value of the locating chromatic numbers of $P_m+P_n$, $K_m+P_n$, $P_m+C_n$, $K_m+C_n$, and $C_m+C_n$ in terms of $m$ and $n$.

\section{The neighbor locating chromatic number}

The following parameter can be defined for disconnected graphs.
\begin{defi}
Let $f$ be a proper $k$-coloring of a graph $G$. If for each two distinct vertices $u$ and $v$ with the same color $f(N_G(u))\neq f(N_G(v))$, then we say $f$ is a {\bf neighbor locating coloring} of $G$. The {\bf neighbor locating chromatic number},
$\Cchi_{{}_{L2}}(G)$,  is the minimum number of colors in a neighbor locating
coloring of $G$.
\end{defi}

Note that we always have $\Cchi(G)\leq\Cchi_{{}_{L2}}(G)\leq |V(G)|$. To see the relation between two parameters $\Cchi_{{}_{L}}$ and $\Cchi_{{}_{L2}}$, let $f$ be a $k$-coloring of the connected graph $G$ and $\Pi=(V_1,V_2,\ldots,V_k)$ be an ordered partition of $V(G)$ into
the resulting color classes. Now for each $v\in V(G)$ determine the color code $c_{{}_{\Pi}}(v)$. Then, in the color code of each vertex replace by $2$ all of the coordinates which are at least two. We call these new color codes {\it modified color codes}. Thus, in the modified color code of a vertex $v$ exactly one coordinate is zero, $|N_G(v)|$ coordinates are $1$, and the other coordinates are $2$.

Now it is easy to see that $f$ is a neighbor locating coloring if and only if different vertices of $G$ have different modified color codes. Therefore, each neighbor locating coloring of $G$ is a locating coloring. Hence, $\Cchi_{{}_L}(G)\leq\Cchi_{{}_{L2}}(G)$.
Also, note that when the diameter of $G$ is at most two, then each locating coloring of $G$ is a neighbor locating coloring and hence, $\Cchi_{{}_{L2}}(G)\leq\Cchi_{{}_L}(G)$. Thus, we have the following proposition.

\begin{prop} \label{tasavi}
If $G$ is a connected graph with diameter at most two, then $\Cchi_{{}_L}(G)=\Cchi_{{}_{L2}}(G)$.
\end{prop}

Specially, $\Cchi_{{}_{L2}}(K_n)=n$ and $\Cchi_{{}_{L2}}(K_{m,n})=m+n$. These two parameters can also be arbitrary far apart. For example, by Theorem \ref{xl2p} (see Section 3), for each $n\geq3$,  $\Cchi_{{}_{L2}}(P_n)=\min\{k~|~k\in \Bbb{N},~n\leq {1\over 2}(k^3-k^2)\}$ while $\Cchi_{{}_L}(P_n)=3$. Note that the path $P_9$ is a graph whose diameter is greater than two but $\Cchi_{{}_{L2}}(P_9)=\Cchi_{{}_{L}}(P_9)=3$. Theorem \ref{xl2p} also implies that for each positive integer $m$ there exists a graph (a path) whose neighbor locating chromatic number is $m$.

\section{The locating chromatic number and the join operation}

In this section, we first study the locating
chromatic number of the join of two arbitrary graphs. Then, we determine the locating chromatic numbers of the friendship graphs, and the join of paths, cycles and complete graphs. Specially, we determine $\Cchi_{{}_L}(F_n)$ and $\Cchi_{{}_L}(W_n)$.
\begin{thm} \label{xlxl2}
For two arbitrary graphs $G_1$ and $G_2$, we have $\Cchi_{{}_L}(G_1+G_2)=\Cchi_{{}_{L2}}(G_1)+\Cchi_{{}_{L2}}(G_2)$.
\end{thm}
\begin{proof}{The diameter of $G_1+G_2$ is at most two and hence, by Proposition \ref{tasavi}, $\Cchi_{{}_{L}}(G_1+G_2)=\Cchi_{{}_{L2}}(G_1+G_2)$. Let $k_1:=\Cchi_{{}_{L2}}(G_1)$, $k_2:=\Cchi_{{}_{L2}}(G_2)$, and $k:=\Cchi_{{}_{L2}}(G_1+G_2)$. Also, let $f$ be a neighbor locating $k$-coloring of $G_1+G_2$. Vertices of $G_1$ are adjacent to the vertices of $G_2$ and hence, $$\{f(u)~|~u\in V(G_1)\}\bigcap\{f(v)~|~v\in V(G_2)\}=\emptyset.$$ Let $k'_1:=|\{f(u)~|~u\in V(G_1)\}|$ and $k'_2:=|\{f(v)~|~v\in V(G_2)\}|$. Thus, $k=k'_1+k'_2$. Assume that $u$ and $u'$ are two vertices of $G_1$ with the same color. Since $f$ is a neighbor locating coloring and $V(G_2)\subseteq N_{{}_{G_1+G_2}}(u)\bigcap N_{{}_{G_1+G_2}}(u')$, we have $f(N_{{}_{G_1}}(u))\neq f(N_{{}_{G_1}}(u'))$. This means that the restriction of $f$ on $V(G_1)$ is a neighbor locating $k'_1$-coloring of $G_1$. Hence, $k_1\leq k'_1$. A similar argument holds for $G_2$. Thus, $k_1+k_2\leq k'_1+k'_2=k$.\\
Now let $f_1$ be a neighbor locating $k_1$-colorings of $G_1$ with the color set $\{1,2,\ldots,k_1\}$, and $f_2$  be a neighbor locating $k_2$-colorings of $G_2$ with the color set $\{k_1+1,k_1+2,\ldots,k_1+k_2\}$. Define a $(k_1+k_2)$-coloring $f'$ of $G_1+G_2$ as $f'(u)=f_1(u)$ when $u\in V(G_1)$, and $f'(v)=f_2(v)$ when $v\in V(G_2)$. If $z_1$ and $z_2$ are two vertices in $G_1+G_2$ with $f'(z_1)=f'(z_2)$, then $\{z_1,z_2\}\subseteq V(G_1)$ or $\{z_1,z_2\}\subseteq V(G_2)$. Without loss of generality, assume that $\{z_1,z_2\}\subseteq V(G_1)$. Since $f_1$ is a neighbor locating $k_1$-coloring and $V(G_2)\subseteq N_{{}_{G_1+G_2}}(z_1)\bigcap N_{{}_{G_1+G_2}}(z_2)$, we have $f'(N_{{}_{G_1+G_2}}(z_1))\neq\ f'(N_{{}_{G_1+G_2}}(z_2))$. This means that  $f'$ is a neighbor locating $(k_1+k_2)$-coloring of $G_1+G_2$ and  hence, $k\leq k_1+k_2$. Thus, $k=k_1+k_2$, which completes the proof.
}\end{proof}

Theorem \ref{xlxl2} and Proposition \ref{tasavi} imply the following corollary.

\begin{cor} \label{diamtasavi}
If $G_1$ and $G_2$ are two connected graphs with diameter at most two, then
$\Cchi_{{}_L}(G_1+G_2)=\Cchi_{{}_{L}}(G_1)+\Cchi_{{}_{L}}(G_2)$.
\end{cor}

Let $m,t$ be two positive integers and $G:=tK_m$ be the graph consisting of $t$ disjoint copies of $K_m$. A coloring of $G$ is a neighbor locating coloring if and only if no two different components of $G$ have the same color set. For a positive integer $k$, the set $[k]$ has ${k\choose m}$ distinct subsets of size $m$. Thus, $\Cchi_{{}_{L2}}(G)=\min\{k~|~t\leq {k\choose m}\}$. Now Theorem \ref{xlxl2} implies the following result.
\begin{prop}
For a positive integer $t$, let $n:=2t+1$. Then, the locating chromatic number of the friendship graph $Fr_n$ is $1+\min\{k~|~t\leq {k\choose 2}\}$.
\end{prop}
Let $P_n=v_1v_2\cdots v_n$ be a path with vertex set $\{v_1,v_2,\ldots,v_n\}$ and edge set $\{v_1v_2,v_2v_3,\ldots,v_{n-1}v_n\}$, and $C_n=v_1v_2\cdots v_nv_1$ be a cycle with vertex set $\{v_1,v_2,\ldots,v_n\}$ and edge set $\{v_1v_2,v_2v_3,\ldots,v_{n-1}v_n,v_nv_1\}$. Let $G\in\{P_n,C_n\}$. Each coloring $f$ of $G$ can be represented by a sequence, say $[f(v_1),f(v_2),\ldots,f(v_n)]$. For convince, we identify each coloring with its sequence and work with the colors instated of vertices. For $1\leq n_1\leq n$, let $f_{|_{[n_1]}}:=[f(v_1),f(v_2),\ldots,f(v_{n_1})]$ be the restriction of $f$ on the subgraph induced by the vertices $\{v_1,v_2,\ldots,v_{n_1}\}$.

If there exists a vertex $v_{{}_l}\in V(G)$ such that $f(v_{{}_l})=s$ and $f(N_{G}(v_{{}_l}))=\{r,t\}$, then we say that the segment $[[r,s,t]]$ occurs in (the corresponding sequence of) $f$. This notation indicates that in $G$ there exists a vertex with color $s$ between two vertices with colors $r$ and $t$. Note that $[[r,s,t]]=[[t,s,r]]$. Also, if $f(v_{{}_l})=s$ and $f(N_{G}(v_{{}_l}))=\{r\}$, then we say that the segment $[[r,s,r]]$ occurs in $f$. This indicates that there exists a vertex with color $s$ between two vertices with color $r$, or there exists a vertex of degree one (a leaf) with color $s$ whose neighbor has color $r$. When $r,s,t$ are elements of $[k]$ with $r\neq s$ and $t\neq s$, then we say that $[[r,s,t]]$ is a possible segment over the set $[k]$. Using these notations we have the following observation.

\begin{obs}
Let $f$ be a $k$-coloring of $P_n$ or $C_n$. Then, $f$ is a neighbor locating $k$-coloring if and only if each possible segment over the set $[k]$ occurs at most once in $f$.
\end{obs}

Now assume that $f$ is a neighbor locating $k$-coloring of $G$, $G\in\{P_n,C_n\}$, for some $k\in \Bbb{N}$. If $u$ and $v$ are two vertices in $G$ with the same color $i$, $1\leq i\leq k$, then $f(N_{G}(u))\neq f(N_{G}(v))$. Note that for each $u\in V(G)$, $|f(N_G(u))|\leq2$. Hence, we have
$$|\{u~|~u\in V(G),~ f(u)=i,~|f(N_{G}(u))|=1\}|\leq k-1$$ and,
$$|\{u~|~u\in V(G),~ f(u)=i,~|f(N_{G}(u))|=2\}|\leq {k-1\choose 2}.$$
This means that there are at most $(k-1)+{k-1\choose 2}={1\over2}(k^2-k)$ vertices in $G$ with color $i$. Hence, $n\leq k({k^2-k\over 2})$. Therefore, we have the following proposition.

\begin{prop} \label{npncn}
Let $n,k$ be two positive integers. If there exists a neighbor locating $k$-coloring $f$ of $P_n$ or $C_n$, then $n\leq {1\over2}(k^3-k^2)$. The equality holds if and only if each possible segment over the set $[k]$ occurs exactly once in $f$.
\end{prop}

When $f=[f(v_1),f(v_2),\ldots,f(v_t)]$ is a coloring of $P_{t}=v_1v_2\cdots v_t$, $f'=[f'(v'_1),f'(v'_2),\ldots,f'(v'_{t'})]$ is a coloring of $P_{t'}=v'_1v'_2\cdots v'_{t'}$, $\{v_1,v_2,\ldots,v_t\}\bigcap\{v'_1,v'_2,\ldots,v'_{t'}\}=\emptyset$, and $f(v_t)\neq f'(v'_1)$, then by $f\oplus f'$ we mean $$[f(v_1),f(v_2),\ldots,f(v_t),f'(v'_1),f'(v'_2),\ldots,f'(v'_{t'})],$$ which is a coloring of the path $P_{t+t'}:=v_1v_2\cdots v_tv'_1v'_2\cdots v'_{t'}$. In fact, we stick the colorings of two small paths in order to get a coloring of a larger path. Note that the segment corresponding to $v_t$ in $f$ is $[[f(v_{t-1}),f(v_t),f(v_{t-1})]]$, while the segment corresponding to $v_t$ in $f\oplus f'$ is $[[f(v_{t-1}),f(v_t),f'(v'_1)]]$. In this case we say that the segment $[[f(v_{t-1}),f(v_t),f'(v'_1)]]$ occurs between $f$ an $f'$. A similar argument holds for $v'_1$. For convince, for the empty sequence $\emptyset$ define $f\oplus\emptyset=\emptyset\oplus f=f$.

Now we are ready to determine the neighbor locating chromatic numbers of paths.

\begin{thm} \label{xl2p}
For a positive integer $n\geq 2$, $\Cchi_{{}_{L2}}(P_n)=m$, where $m:=\min\{k~|~k\in \Bbb{N},~n\leq {1\over2}(k^3-k^2)\}$. Particularly, there exists a neighbor locating $m$-coloring $f_n$ of the path $P_n=v_1v_2\cdots v_n$ such that  $f_n(v_{n-1})=2$ and $f_n(v_n)=1$. For $n\geq 9$, $f_n(v_{n-2})=m$. Moreover, for $n\geq 9$ and $n\neq {1\over 2}(m^3-m^2)-1$,  $f_n(v_1)=2$ and $f_n(v_2)=1$.
\end{thm}
\begin{proof}{Since $${1\over2}((m-1)^3-(m-1)^2)<n\leq {1\over2}(m^3-m^2),$$ if we give a neighbor locating $m$-coloring of $P_n$, then Proposition \ref{npncn} implies that $\Cchi_{{}_{L2}}(P_n)=m$.
\\For $2\leq n\leq 50$, consider the colorings which are listed in Table 1. It is not hard to check that each possible segment over the set $[5]$ occurs at most once in the $f_i,~2\leq i\leq 50$. Hence, each $f_i$ is a neighbor locating coloring. Note that ${3^3-3^2\over2}=9$, ${4^3-4^2\over2}=24$, and ${5^3-5^2\over2}=50$. Also, note that all of the possible segments over the sets $[3]$, $[4]$, and $[5]$ occur in $f_9$, $f_{24}$, and $f_{50}$, respectively.
\vspace*{-3mm}
\begin{center}
\begin{table}[h]\label{t1}
{\caption {Optimal neighbor locating colorings for the small paths.}}
\begin{tabular}{|l|l|}
 \hline
$f_2:=[2,1]$ &  $f_{19}:=f_9\oplus [4,1,4,3,4,3,2,4,2,1]$ \\
$f_3:=[3,2,1]$ &  $f_{20}:=f_9\oplus [4,3,1,4,2,4,3,2,4,2,1]$ \\
$f_4:=[1,3,2,1]$ &  $f_{21}:=f_9\oplus [4,1,4,2,4,3,4,3,2,4,2,1]$ \\
$f_5:=[2,1,3,2,1]$ &  $f_{22}:=f_7\oplus [4,3,4,2,4,1,4,1,3,4,3,2,4,2,1]$ \\
$f_6:=[3,2,3,1,2,1]$ &  $f_{23}:=f_8\oplus[4,3,4,2,4,1,4,1,3,4,3,2,4,2,1]$ \\
$f_7:=[2,1,3,2,3,2,1]$ &  $f_{24}:=f_9\oplus [4,3,4,2,4,1,4,1,3,4,3,2,4,2,1]$ \\
$f_8:=[3,2,3,1,3,1,2,1]$ &  $f_{25}:=f_{{24}_{|_{[22]}}}\oplus [5,2,1]$ \\
$f_9:=[2,1,3,1,3,2,3,2,1]$ &  $f_{26}:=f_{{24}_{|_{[22]}}}\oplus [2,5,2,1]$ \\
$f_{10}:=[2,1,3,1,3,2,3,4,2,1]$ &  $f_{27}:=f_{{24}_{|_{[22]}}}\oplus [2,1,5,2,1]$ \\
$f_{11}:=[2,1,3,1,3,2,3,2,4,2,1]$ & $f_{28}:=f_{{24}_{|_{[22]}}}\oplus [5,3,1,5,2,1]$ \\
$f_{12}:=[2,1,3,1,3,2,3,2,1,4,2,1]$ & $f_{29}:=f_{{24}_{|_{[22]}}}\oplus [5,3,5,1,5,2,1]$ \\
$f_{13}:=[2,1,3,1,3,2,3,4,3,1,4,2,1]$ & $f_{30}:=f_{{24}_{|_{[22]}}}\oplus [5,3,5,1,3,5,2,1]$ \\
$f_{14}:=[2,1,3,1,3,2,3,4,3,4,1,4,2,1]$ &  $f_{31}:=f_{{24}_{|_{[22]}}}\oplus [5,3,5,1,5,2,5,2,1]$ \\
$f_{15}:=[2,1,3,1,3,2,3,4,3,4,1,3,4,2,1]$ & $f_{32}:=f_{{24}_{|_{[22]}}}\oplus [5,3,5,1,3,5,2,5,2,1]$ \\
$f_{16}:=[2,1,3,1,3,2,3,4,3,4,1,4,2,4,2,1]$ & $f_{33}:=f_{24}\oplus[5,1,3,5,3,2,5,2,1]$ \\
$f_{17}:=[2,1,3,1,3,2,3,4,3,4,1,3,4,2,4,2,1]$ & $f_{34}:=f_{24}\oplus[5,1,5,3,5,3,2,5,2,1]$ \\
$f_{18}:=F_9\oplus [4,1,3,4,3,2,4,2,1]$ & \\
\hline
\end{tabular}
\begin{tabular}{|l|}
$f_{26+i}:=f_{i}\oplus[5,4,5,3,5,2,5,1,5,3,4,5,4,2,5,4,1,5,1,3,5,3,2,5,2,1], ~~~9\leq i\leq 24\!$~~~~~~~~~~~\\
\hline
\end{tabular}
\end{table}
\end{center}
\vspace*{-7mm}
Here after let $n\geq 51$. Thus, $m\geq 6$. Now in an inductive way we prove the theorem. Let $n':={1\over2}((m-1)^3-(m-1)^2)$, and assume that $f_{n'}$ is a neighbor locating $(m-1)$-coloring of $P_{n'}$ with the mentioned properties in the theorem (let us to denote this by writing $f_{n'}=[2,1,\ldots,m-1,2,1]$). Specially, by Proposition \ref{npncn}, all of the possible segments over the set $[m-1]$ occur in $f_{n'}$. Note that ${1\over2}(m^3-m^2)=n'+(2(m-1)+3{m-1\choose2})$. Using the new color $``m"$, we will add $2(m-1)+3{m-1\choose2}$ new entries to $f_{n'}$. These new entries are $(m-1)$ pairs of the form $[m,i]$, and ${m-1\choose2}$ triples of the form $[m,i,j]$, $\{i,j\}\subseteq [m-1]$. Step by step, we provide a neighbor locating $m$-coloring $f_{i}$ for each $n'<i\leq n'+(2(m-1)+3{m-1\choose 2})$. In each step we modify the coloring for a path with one more vertex. Equivalently, we add a new entry to some where in the coloring sequence and probably, we change some other entries.

Let $T:=[m,1,3,m,3,2,m,2,1]$ and $A:=[m,m-4,m,m-5,\ldots,m,2,m,1]$. For each $i,j$, $4\leq i\leq m-1$, and $1\leq j\leq i-2$, let
$D_{i,j}:=[m,i,j,m,i,j-1,\ldots,m,i,1]$. Also, let $D_i:=[m,i-1,i]\oplus D_{i,i-2}$ and $D_{[i]}:=D_i\oplus D_{i-1}\oplus\cdots\oplus D_4$. For example we have $D_5=[m,4,5,m,5,3,m,5,2,m,5,1]$. For convince, define $D_{i,0}=D_3=D_{[3]}=\emptyset$.
Now consider the following coloring which is an $m$-coloring of a path with $n'+2(m-1)+3{m-1\choose 2}$ vertices.
$$f_{n'+2(m-1)+3{m-1\choose2}}=f_{n'}\oplus[m,m-1,m,m-2,m,m-3]\oplus A\oplus D_{[m-1]}\oplus T.$$
This is our final ``complete model". Using this complete model we want to build the smaller colorings $\{f_{n'+i}~|~1\leq i<2(m-1)+3{m-1\choose 2}\}$. Note that all of the possible segments over the set $[m]$ occur in $f_{n'+2(m-1)+3{m-1\choose2}}$, each of them just once. More precisely,
\begin{itemize}
\item All of the possible segments over the set $[m-1]$ occur in $f_{n'}$, except the segment $[[2,1,2]]$ which occurs at the end of $T$.
\item The segments of the form $[[m,i,j]]$, where $i,j\in [m-1]$ and $i\neq j$, occur in $D_{[m-1]}\oplus T$, except the segment $[[m,1,2]]=[[2,1,m]]$ which occurs between $f_{n'}$ and $[m,m-1,m,m-2,m,m-3]$.
\item The segments of the form $[[m,i,m]]$, $2\leq i\leq m-1$, occur in $[m,m-1,m,m-2,m,m-3]\oplus A$. The segment $[[m,1,m]]$ occurs between $A$ and $D_{[m-1]}$.
\item The segments of the form $[[i+1,m,i]]$, $1\leq i\leq m-2$, occur in $[m,m-1,m,m-2,m,m-3]\oplus A$.
\item The segments of the forms $[[j,m,i]]$ and $[[i,m,i]]$, where $4\leq i\leq m-1$ and $2\leq j\leq i-2$, occur in $D_i$, inside $D_{[m-1]}$.
\item The segments of the form $[[1,m,j]]$, $3\leq j\leq m-3$, occur between $D_{j+2}$ and $D_{j+1}$, inside $D_{[m-1]}$. The segment $[[1,m,1]]$ occurs between $D_{[m-1]}$ and $T$. The segment $[[1,m,m-2]]$ occurs between $A$ and $D_{[m-1]}$, and the segment $[[1,m,m-1]]$ occurs between $f_{n'}$ and $[m,m-1,m,m-2,m,m-3]$.
\item The segments of the form $[[2,m,j]]$, $4\leq j\leq m-1$, occur in $D_j$. The segment $[[2,m,2]]$ occurs in $T$.
\item The segments of the form $[[3,m,j]]$, $5\leq j\leq m-1$, occur in $D_j$. The segment $[[3,m,3]]$ occurs in $T$.
\end{itemize}
Note that $f_{24}$ and $f_{50}$ are given using this complete model. Now we proceed to build the other smaller colorings. Note that by the hypothesis, we have $f_{n'}=[2,1,\ldots,m-1,2,1]$. Let
\begin{itemize}
\item[] $f_{n'+1}:=f_{{n'}_{|_{[n'-2]}}}\oplus [m,2,1]$,
\item[] $f_{n'+2}:=f_{{n'}_{|_{[n'-2]}}}\oplus [2,m,2,1]$,
\item[] $f_{n'+3}:=f_{{n'}_{|_{[n'-2]}}}\oplus [2,1,m,2,1]$,
\item[] $f_{n'+4}:=f_{{n'}_{|_{[n'-2]}}}\oplus [m,3,1,m,2,1]$,
\item[] $f_{n'+5}:=f_{{n'}_{|_{[n'-2]}}}\oplus [m,3,m,1,m,2,1]$,
\item[] $f_{n'+6}:=f_{{n'}_{|_{[n'-2]}}}\oplus [m,3,m,1,3,m,2,1]$,
\item[] $f_{n'+7}:=f_{{n'}_{|_{[n'-2]}}}\oplus [m,3,m,1,m,2,m,2,1]$,
\item[] $f_{n'+8}:=f_{{n'}_{|_{[n'-2]}}}\oplus [m,3,m,1,3,m,2,m,2,1]$,
\item[] $f_{n'+9}:=f_{{n'}_{|_{[n'-2]}}}\oplus [2,1,m,1,3,m,3,2,m,2,1]$,
\item[] $f_{n'+10}:=f_{{n'}_{|_{[n'-2]}}}\oplus [2,1,m,1,m,3,m,3,2,m,2,1]$,
\item[] $f_{n'+11}:=f_{{n'}_{|_{[n'-2]}}}\oplus [2,1,m,m-1,1,m,3,m,3,2,m,2,1]$,
\item[] $f_{n'+12}:=f_{{n'}_{|_{[n'-2]}}}\oplus [2,1,m,m-1,m-2,m,1,3,m,3,2,m,2,1]$.
\end{itemize}

Let $1\leq i\leq12$. The coloring $f_{n'+i}$ has two parts. The first part is $f_{{n'}_{|_{[n'-2]}}}$ which the color $m$ does not appear in it, and the second part which $m$ appears in it. Since $f_{n'}$ is a neighbor locating $(m-1)$-coloring, each possible segment over the set $[m-1]$ occurs at most once in $f_{{n'}_{|_{[n'-2]}}}$. Note that the segment $[[2,1,2]]$ occurs at the end of the second pat of $f_{n'+i}$ and not in the first part. Also, it is easy to see that each segment in $f_{n'+i}$ which contains $m$ occurs just once. Hence, $f_{n'+i}$ is a neighbor locating $m$-coloring. Note that $f_{n'+12}=f_{n'}\oplus [m,m-1,m-2]\oplus T$.
Now step by step we add the part $A$. Let
$$f_{n'+12+1}:=f_{n'}\oplus [m,m-1,m,m-2]\oplus T,~~~f_{n'+12+2}:=f_{n'}\oplus [m,m-1,m-2]\oplus [m,1]\oplus T.$$
Also, for each $2\leq i\leq m-4$, let
\begin{eqnarray*}
f_{n'+12+2i-1}:=f_{n'}\oplus [m,m-1,m,m-2]\oplus [m,i-1,m,i-2,\ldots,m,1] \oplus T,
\end{eqnarray*}
and
\begin{eqnarray*}
f_{n'+12+2i}:=f_{n'}\oplus [m,m-1,m-2]\oplus [m,i,m,i-1,\ldots,m,1]\oplus T.
\end{eqnarray*}

Specially, $f_{n'+12+2(m-4)}=f_{n'}\oplus [m,m-1,m-2]\oplus A\oplus T$.
Let $1\leq j\leq 2(m-4)$. In the coloring $f_{n'+12+j}$ the segment $[[2,1,2]]$ occurs at the end of part $T$ instead of part $f_{n'}$. Note that in $f_{n'+12+j}$, the segment corresponding to the final entry of $f_{n'}$ is $[[2,1,m]]$, not $[[2,1,2]]$. Each possible segment over the set $[m-1]$ occurs at most once and, except $[[2,1,2]]$, each one occurs just in the part $f_{n'}$. Also, by the case by case investigation, we can see that each possible segment containing $m$ occurs at most once. Hence, $f_{n'+12+j}$ is a neighbor locating $m$-coloring.

For adding the parts $D_4$, $D_5$, $\ldots$, $D_{m-3}$ we proceed as follow. Let $4\leq i\leq m-3$ and assume that $D_{i-1}$ is added (note that $D_3=\emptyset$). For adding $D_i$, alternately, we add a new entry $m$, then we remove it in order to add the portion $[m,m-3]$ to the beginning of  $A$, and then we remove this portion in order to add a portion of the form $[m,i,j]$. More precisely, assume that $D_{i,j-1}$ is completed, where $1\leq j\leq i-1$. We want to add the portion $[m,i,j]$ or $[m,j,i]$ of $D_i$. Let $n_i:=n'+12+2(m-4)+3(3+4+\cdots+(i-1-1))$. Note that $n_4=n+12+2(m-4)$ and $D_{i,0}=D_{[3]}=\emptyset$. Let
\begin{eqnarray*}
&f_{n_i+3j-2}:=f_{n'}\oplus [m,m-1,m,m-2]\oplus A\oplus D_{i,j-1}\oplus D_{[i-1]}\oplus T,&\\
&f_{n_i+3j-1}:=f_{n'}\oplus [m,m-1,m-2]\oplus [m,m-3] \oplus A\oplus D_{i,j-1}\oplus D_{[i-1]}\oplus T,&
\end{eqnarray*}
and
\begin{eqnarray*}
f_{n_i+3j}:=\left\{ \begin{array}{ll} f_{n'}\oplus [m,m-1,m-2]\oplus A\oplus [m,i,j]\oplus D_{i,j-1}\oplus D_{[i-1]}\oplus T &  ~~j<i-1 \\ f_{n'}\oplus [m,m-1,m-2]\oplus A\oplus [m,j,i]\oplus D_{i,j-1}\oplus D_{[i-1]}\oplus T & ~~j=i-1. \end{array} \right.
\end{eqnarray*}
Except the segment $[[2,1,2]]$ which occurs at the end of $T$, all of the other possible segments over the set $[m-1]$ occur just in $f_{n'}$. Also, by considering the structures of $A,~D_{i,j-1},~D_{[i-1]}$ and $T$ (and similar to what we said about the complete model) it is not hard to see that each segment containing $m$ occurs at most once in this colorings. Hence, these colorings are neighbor locating $m$-colorings.

Let $n'':=n_{m-3}+3(m-4)$. Since $(n'+2(m-1)+3{m-1\choose2})-n''=6m-12$, we need $6m-12$ steps to complete the proof. Adding $D_{m-2}$ and $D_{m-1}$ is possible but it is complicated and requires more details. Instead, we use the completed model, just we replace $f_{n'}$ with the smaller colorings.
let
\begin{eqnarray*}
f_{n''+j}:=f_{n'-6m+12+j}\oplus[m,m-1,m,m-2,m,m-3]\oplus A\oplus D_{[m-1]}\oplus T,
\end{eqnarray*}
where $1\leq j\leq 6m-12$. Note that since $m\geq6$, $n'-(6m-12)\geq9$.
}\end{proof}

Theorems \ref{xlxl2} and \ref{xl2p} imply the following two corollaries.

\begin{cor}\label{kmpn}
For $m\geq1$ and $n\geq 2$, we have $\Cchi_{{}_{L}}(K_m+P_n)=m+\min\{k~|~k\in \Bbb{N},~n\leq {1\over 2}(k^3-k^2)\}$.
Specially, the locating chromatic number of the fan $F_n$ is $\Cchi_{{}_{L}}(K_1+P_n)$.
\end{cor}

\begin{cor} \label{pnpn}
For two positive integers $m\geq 2$ and $n\geq2$, let $m_{{}_0}:=\min\{k~|~k\in\Bbb{N},~m\leq{1\over2}(k^3-k^2)\}$ and $n_{{}_0}:=\min\{k~|~k\in\Bbb{N},~n\leq{1\over2}(k^3-k^2)\}$. Then,
$\Cchi_{{}_L}(P_m+P_n)=m_{{}_0}+n_{{}_0}$.
\end{cor}

Now we determine the neighbor locating chromatic numbers of the cycles. Then using it we determine the exact values of $\Cchi_{{}_L}(P_m+C_n)$, $\Cchi_{{}_L}(K_m+C_n)$, and $\Cchi_{{}_L}(C_m+C_n)$.

For each $n$, $3\leq n<9$, consider the following coloring (sequence) $h_n$ of the cycle $C_n$.
\\$h_3:=[1,2,3],~h_4:=[1,2,3,4],~h_5:=[1,2,1,2,3],~h_6:=[1,2,1,3,2,4],~h_7:=[2,1,3,2,3,2,1],$
\\$h_8:=[3,2,3,1,3,1,2,1,4]$.
\\It is easy to check that each coloring $h_n$ is a neighbor locating coloring. Note that $\Cchi_{{}_L}(C_n)$ is three or four depending on the parity of $n$, and $\Cchi_{{}_L}(C_n)\leq\Cchi_{{}_{L2}}(C_n)$. Therefore, $\Cchi_{{}_L}(C_n)=\Cchi_{{}_{L2}}(C_n)$ for $3\leq n<9$.
For the general case $n\geq9$ we have the following theorem.

\begin{thm} \label{xl2c}
For a positive integer $n\geq 9$, let $n_{{}_0}:=\min\{k~|~k\in \Bbb{N},~n\leq {1\over2}(k^3-k^2)\}$. Then,
\begin{eqnarray*}
\Cchi_{{}_{L2}}(C_n)=\left\{\begin{array}{ll}  n_{{}_0} &~~n\neq {1\over2}(n_{{}_0}^3-n_{{}_0}^2)-1 \\  n_{{}_0}+1 &~~n={1\over2}(n_{{}_0}^3-n_{{}_0}^2)-1.\end{array}\right.
\end{eqnarray*}
\end{thm}
\begin{proof}{Suppose that $C_n=v_1v_2\cdots v_nv_1$. By Proposition \ref{npncn}, we have $\Cchi_{{}_{L2}}(C_n)\geq n_{{}_0}$. First assume that $n\neq {1\over2}(n_{{}_0}^3-n_{{}_0}^2)-1$. By Theorem 2, there exists a neighbor locating $n_{{}_0}$-coloring $f_n$ of the path $P_n:=v_1v_2\cdots v_n$ such that $f_n(v_1)=2$, $f_n(v_2)=1$, $f_n(v_{n-1})=2$, and $f_n(v_n)=1$. Consider $f_n$ as a coloring of the vertices of $C_n$. Since $f_n(v_1)\neq f_n(v_n)$, this is a proper coloring of $C_n$. Note that $E(C_n)=E(P_n)\bigcup \{v_nv_1\}$. Hence, for each $1\leq i\leq n$, we have $f_n(N_{C_n}(v_i))=f_n(N_{P_n}(v_i))$. Therefore, $f_n$ is also a neighbor locating $n_{{}_0}$-coloring of $C_n$. This implies that $\Cchi_{{}_{L2}}(C_n)=n_{{}_0}$.

Now assume that $n:={1\over2}(n_{{}_0}^3-n_{{}_0}^2)-1$. By Theorem 2, there exists a neighbor locating $n_{{}_0}$-coloring $f_{n-1}$ of the path $P_{n-1}=v_1v_2\cdots v_{n-1}$ such that $f_{n-1}(v_1)=2$  and $f_{n-1}(v_{n-1})=1$. Define the coloring $f'_n$ of $C_n$ as $f'_n(v_n)=n_{{}_0}+1$ and $f'_n(v_i)=f_{n-1}(v_i)$ for $1\leq i\leq n-1$. Note that $n_{{}_0}+1\in f'_n(N_{C_n}(v_1))\bigcap f'_n(N_{C_n}(v_{n-1}))$, $f'_n(v_1)\neq f'_n(v_{n-1})$, and $f'_n(N_{C_n}(v_i))=f_{n-1}(N_{P_{n-1}}(v_i))$ for each $2\leq i\leq n-2$. Thus, $f'_n$ is a neighbor locating $(n_{{}_0}+1)$-coloring of $C_n$. Hence, $\Cchi_{{}_{L2}}(C_n)\leq n_{{}_0}+1$.

We want to show $\Cchi_{{}_{L2}}(C_n)\neq n_{{}_0}$. Suppose on the contrary there exists a neighbor locating $n_{{}_0}$-coloring $f$ of $C_n$. For each $1\leq i\leq n_{{}_0}$, let $V_i:=\{x~|~x\in V(C_n),~f(x)=i\}$. Since $f$ is a neighbor locating $n_{{}_0}$-coloring, each color class contains at most ${1\over2}(n_{{}_0}^2-n_{{}_0})$ vertices (see the argument before Proposition \ref{npncn}). Now since $n={1\over2}(n_{{}_0}^3-n_{{}_0}^2)-1$, exactly one of the color classes, say $V_1$, has size ${1\over2}(n_{{}_0}^2-n_{{}_0})-1$ and the others have size ${1\over2}(n_{{}_0}^2-n_{{}_0})$. For each $2\leq i\leq n_{{}_0}$, let $X_i:=\{(x,y)~|~x\in N_{C_n}(y),~f(x)=1,~f(y)=i\}$. Let $2\leq i\leq n_{{}_0}$. Since $|V_i|={1\over2}(n_{{}_0}^2-n_{{}_0})$, all of the possible segments of the form $[[r,i,j]]$, where $r\in [n_{{}_0}]$ and $j\in [n_{{}_0}]$, occur in $f$. Thus, for each $j$ with $j\notin\{1,i\}$, there exists $y\in V_i$ such that $f(N_{C_n}(y))=\{1,j\}$. Also, there exists $z\in V_i$ such that $f(N_{C_n}(z))=\{1\}$. This implies that $|X_i|=(n_{{}_0}-2)+2=n_{{}_0}$. Hence, $|X|=(n_{{}_0}-1)n_{{}_0}$, where $X:=X_2\bigcup X_3\bigcup\cdots\bigcup X_{n_{{}_0}}$. Each vertex $x$ with color $1$ has two neighbors and hence, $|X|=2~|\{x~|~x\in V(C_n),~f(x)=1\}|$. This means that there are ${|X|\over2}={(n_{{}_0}-1)n_{{}_0}\over2}$ vertices with color $1$, which is a contradiction.
}\end{proof}

Theorems \ref{xlxl2} and \ref{xl2c} imply the following corollaries.

\begin{cor} \label{pmcn}
For two positive integers $m\geq 2$ and $n\geq 3$, let $m_{{}_0}:=\min\{k~|~k\in \Bbb{N},~m\leq {1\over2}(k^3-k^2)\}$ and $n_{{}_0}:=\min\{k~|~k\in \Bbb{N},~n\leq {1\over2}(k^3-k^2)\}$. Then,
\begin{eqnarray*}
\Cchi_{{}_{L}}(P_m+C_n)=\left\{\begin{array}{ll} m_{{}_0}+\Cchi_{{}_L}(C_n) &~~3\leq n<9 \\ m_{{}_0}+n_{{}_0} &~~n\geq9,~n\neq {1\over2}(n_{{}_0}^3-n_{{}_0}^2)-1\\ m_{{}_0}+n_{{}_0}+1 &~~n\geq9,~n={1\over2}(n_{{}_0}^3-n_{{}_0}^2)-1.\end{array}\right.
\end{eqnarray*}
\end{cor}

\begin{cor} \label{kmcn}
For two positive integers $m\geq 1$ and $n\geq 3$, let $n_{{}_0}:=\min\{k~|~k\in \Bbb{N},~n\leq {1\over 2}(k^3-k^2)\}$. Then,
\begin{eqnarray*}
\Cchi_{{}_{L}}(K_m+C_n)=\left\{\begin{array}{ll} m+\Cchi_{{}_L}(C_n) &~~3\leq n<9\\  m+n_{{}_0} &~~n\geq9,~n\neq {1\over2}(n_{{}_0}^3-n_{{}_0}^2)-1\\ m+n_{{}_0}+1 &~~n\geq9,~n={1\over2}(n_{{}_0}^3-n_{{}_0}^2)-1.\end{array}\right.
\end{eqnarray*}
Specially, the locating chromatic number of the wheel $W_n$ is $\Cchi_{{}_L}(K_1+C_n)$.
\end{cor}

\begin{cor}\label{cmcn}
For positive integers $m$ and $n$, $3\leq m\leq n$, let $m_{{}_0}:=\min\{k~|~k\in \Bbb{N},~m\leq {1\over2}(k^3-k^2)\}$ and $n_{{}_0}:=\min\{k~|~k\in \Bbb{N},~n\leq {1\over2}(k^3-k^2)\}$. Then,
\begin{eqnarray*}
\Cchi_{{}_{L}}(C_m+C_n)=\left\{\begin{array}{ll}
\Cchi_{{}_L}(C_m)+\Cchi_{{}_L}(C_n) &~~n<9\\
\Cchi_{{}_L}(C_m)+n_{{}_0} &~~m<9\leq n,~n\neq {1\over2}(n_{{}_0}^3-n_{{}_0}^2)-1\\
\Cchi_{{}_L}(C_m)+n_{{}_0}+1 &~~m<9\leq n,~n={1\over2}(n_{{}_0}^3-n_{{}_0}^2)-1\\
m_{{}_0}+n_{{}_0} &~~m\geq9,~m\neq {1\over2}(m_{{}_0}^3-m_{{}_0}^2)-1,~n\neq {1\over2}(n_{{}_0}^3-n_{{}_0}^2)-1\\
m_{{}_0}+n_{{}_0}+1 &~~m\geq9,~m={1\over2}(m_{{}_0}^3-m_{{}_0}^2)-1,~n\neq {1\over2}(n_{{}_0}^3-n_{{}_0}^2)-1\\
m_{{}_0}+n_{{}_0}+1 &~~m\geq9,~m\neq {1\over2}(m_{{}_0}^3-m_{{}_0}^2)-1,~n={1\over2}(n_{{}_0}^3-n_{{}_0}^2)-1\\
m_{{}_0}+n_{{}_0}+2 &~~m\geq9,~m={1\over2}(m_{{}_0}^3-m_{{}_0}^2)-1,~n={1\over2}(n_{{}_0}^3-n_{{}_0}^2)-1.
\end{array}\right.
\end{eqnarray*}
\end{cor}
\begin{remark}
Note that the diameter of a complete multipartite graph is two and its locating chromatic number is equal to the number of its vertices. Hence, two corollaries \ref{kmpn} and \ref{kmcn} hold also for complete multipartite graphs (such as stars) instead of complete graphs.
\end{remark}

\end{document}